\documentclass{scrartcl}

\usepackage{amsfonts}
\usepackage{amsmath}
\usepackage{amsopn}
\usepackage{hyperref}
\usepackage{graphicx}
\usepackage{subcaption}
\usepackage{cleveref}

\title{Computing Maxwell eigenmodes with Bloch boundary conditions}

\author{Steffen B\"orm
        \thanks{Mathematisches Seminar,
                Christian-Albrechts-Universit\"at zu Kiel
                (\href{mail:boerm@math.uni-kiel.de}{boerm@math.uni-kiel.de},
                 website \url{https://www.math.uni-kiel.de/scicom})} \and
        Ralf K\"ohl
        \thanks{Mathematisches Seminar,
                Christian-Albrechts-Universit\"at zu Kiel
                (\href{mail:koehl@math.uni-kiel.de}{koehl@math.uni-kiel.de})} \and
        Nahid Talebi
        \thanks{Institut f\"ur Exp. und Ang. Physik,
                Christian-Albrechts-Universit\"at zu Kiel
                (\href{mail:talebi@physik.uni-kiel.de}{talebi@physik.uni-kiel.de})}}

\DeclareMathOperator{\scurl}{curl}
\DeclareMathOperator{\vcurl}{\mathbf{curl}}

\newtheorem{theorem}{Theorem}
\newtheorem{lemma}[theorem]{Lemma}
\newenvironment{proof}{\noindent\emph{Proof:}}{\hfill$\Box$}

\newcommand{\bbbc}{\mathbb{C}}
\newcommand{\bbbr}{\mathbb{R}}
\newcommand{\bbbn}{\mathbb{N}}
\newcommand{\Idx}{\mathcal{I}}
\newcommand{\Jdx}{\mathcal{J}}

\newcommand{\Hbloch}{H_\text{Bloch}(\scurl,\Omega,k)}

\begin{document}

\maketitle

\begin{abstract}
Our goal is to predict the band structure of photonic crystals.
This task requires us to compute a number of the smallest non-zero
eigenvalues of the time-harmonic Maxwell operator depending on the
chosen Bloch boundary conditions.

We propose to use a block inverse iteration preconditioned with a suitably
modified geometric multigrid method.
Since we are only interested in non-zero eigenvalues, we eliminate the
large null space by combining a lifting operator and a secondary
multigrid method.
To obtain suitable initial guesses for the iteration, we employ a generalized
extrapolation technique based on the minimization of the Rayleigh quotient
that significantly reduces the number of iteration steps and allows us to
treat families of very large eigenvalue problems efficiently.
\end{abstract}




\section{Introduction}

A periodic optical medium, in the form of a grating, a multilayered geometry,
a patterned thin film, or a more general three-dimensional configuration, has
various applications in tailoring the properties of light.
Particularly, in analogy with electronic properties of crystalline matters,
the band structure of optical waves, i.e., the dispersion diagram of
individual optical modes in photonic crystals, is of great interest, since
it provides the opportunity to investigate the optical density of states
and the propagation properties of light in the lattice.
Due to the high symmetry of the structure in both real and reciprocal spaces,
the band structures are analyzed in the Voronoi cell, namely the Brillouin
zone.
Particularly, for realizing photonic cavities, slow waveguides, as well as
several platforms for enhanced light-matter interactions, searching for
photonic crystal configurations that offer a global bandgap is attractive.
Moreover, tailoring the phase of the optical waves on the lattice could allow
for exploring topological aspects as well.

The optical waves in a lattice are in the form of so-called Bloch waves,
constituting a plane wave modulated by a periodic function, where the latter
function sustains the periodicity of the lattice.
Thus, for numerically calculating the Bloch waves, it is sufficient to consider
only a unit cell of the lattice, combined with appropriate boundary conditions.

In this article, we focus on the two-dimensional setting,
i.e., we are looking for the smallest non-zero eigenvalues
$\lambda$ and corresponding eigenvectors $u\in H(\scurl,\Omega)$
satisfying the two-dimensional Maxwell equation
\begin{align}\label{eq:maxwell}
  \vcurl \frac{1}{\epsilon(x)} \scurl u(x) &= \lambda u(x) &
  &\text{ for all } x\in\bbbr^2.
\end{align}
Here the scalar- and vector-valued curl operators are given
by
\begin{align*}
  \scurl u(x) &= \partial_2 u_1(x) - \partial_1 u_2(x), &
  \vcurl \varphi(x) &= \begin{pmatrix}
    -\partial_2 \varphi(x)\\
    \partial_1 \varphi(x)
  \end{pmatrix}
\end{align*}
for $u\in H(\scurl,\Omega)$ and $\varphi\in H^1(\Omega)$.
We assume the dielectricity constant $\epsilon$ to be periodic
with period $a\in\bbbr_{>0}$ in the first coordinate and period
$b\in\bbbr_{>0}$ in the second, i.e.,
\begin{align}\label{eq:eps_periodic}
  \epsilon(x_1+a,x_2) &= \epsilon(x_1,x_2)
  = \epsilon(x_1,x_2+b) &
  &\text{ for all } x\in\bbbr^2.
\end{align}
Due to Bloch's theorem \cite{BL29}, the eigenvectors can be represented in
the factorized form
\begin{align*}
  u(x) &= \exp(\iota \langle k,x \rangle) \hat u(x) &
  &\text{ for all } x\in\bbbr^2,
\end{align*}
where $k\in\bbbr^2$ is the Bloch parameter and $\hat u$ is
periodic, i.e.,
\begin{align*}
  \hat u(x_1+a,x_2) &= \hat u(x_1,x_2)
  = \hat u(x_1,x_2+b) &
  &\text{ for all } x\in\bbbr^2.
\end{align*}
Applying these identities yields
\begin{subequations}\label{eq:bloch_shift}
\begin{align}
  u(x_1+a,x_2) &= \exp(\iota k_1 a) \exp(\iota \langle k,x \rangle)
                               \hat u(x_1+a,x_2)\notag\\
           &= \exp(\iota k_1 a) \exp(\iota \langle k,x \rangle)
                            \hat u(x_1,x_2)
            = \exp(\iota k_1 a) u(x_1,x_2),\\
  u(x_1,x_2+b) &= \exp(\iota k_2 b) \exp(\iota \langle k,x \rangle)
                               \hat u(x_1,b)\notag\\
           &= \exp(\iota k_2 b) \exp(\iota \langle k,x \rangle)
                               \hat u(x_1,x_2)
            = \exp(\iota k_2 b) u(x_1,x_2)
\end{align}
\end{subequations}
for all $x\in\bbbr^2$.
Taking advantage of this ``phase-shifted periodicity'' allows
us to restrict our attention to the fundamental domain with
respect to translation
\begin{equation*}
  \Omega := [0,a] \times [0,b]
\end{equation*}
subject to the Bloch boundary conditions
\begin{subequations}\label{eq:bloch_boundary}
\begin{align}
  u_1(x_1,b) &= \exp(\iota k_2 b)\, u_1(x_1,0) &
  &\text{ for all } x_1\in[0,a],\\
  u_2(a,x_2) &= \exp(\iota k_1 a)\, u_2(0,x_2) &
  &\text{ for all } x_2\in[0,b],
\end{align}
\end{subequations}
with the Bloch parameter $k\in[-\pi/a,\pi/a]\times[-\pi/b,\pi/b]$.
We therefore work with the subspace
\begin{equation*}
  \Hbloch
  := \{ u\in L^2(\Omega,\bbbc^2)\ :\ \scurl u\in L^2(\Omega,\bbbc),
            \ u \text{ satisfies } \cref{eq:bloch_boundary} \}
\end{equation*}
of $H(\scurl,\Omega)$.
Since \cref{eq:bloch_boundary} involves only the tangential traces
of $u\in H(\scurl,\Omega)$, this is a closed subspace of a Hilbert space
and therefore itself a Hilbert space.

Multiplying \cref{eq:maxwell} with a test function
$v\in\Hbloch$ and integrating by parts (cf. \cref{eq:partial_integration})
yields the variational formulation
\begin{align}\label{eq:variational}
  a(v,u) &= \lambda\, m(v,u) &
  &\text{ for all } v\in\Hbloch
\end{align}
with the sesquilinear forms
\begin{align*}
  a\colon \Hbloch\times\Hbloch &\to \bbbc, &
  (v,u) &\mapsto \int_\Omega
         \frac{\overline{\scurl v(x)} \scurl u(x)}{\epsilon(x)} \,dx,\\
  m\colon \Hbloch\times\Hbloch &\to \bbbc, &
  (v,u) &\mapsto \int_\Omega \langle v(x), u(x) \rangle \,dx.
\end{align*}
We discretize it using a Galerkin scheme with bilinear
N\'ed\'elec trial and test functions \cite{NE80} to obtain
a finite-dimensional eigenvalue problem
\begin{equation}\label{eq:matrix}
  A e = \lambda M e
\end{equation}
with a stiffness matrix $A\in\bbbc^{n\times n}$ and a mass
matrix $M\in\bbbc^{n\times n}$.
Since the bilinear forms of \cref{eq:variational} are Hermitian,
the matrices $A$ and $M$ are self-adjoint, $A$ is positive semi-definite,
and $M$ is positive definite.
These properties imply that we can find a biorthogonal basis consisting
of eigenvectors of the matrices $A$ and $M$.

When treating this eigenvalue problem numerically, we are faced with
two challenges:
on the one hand, the sesquilinear form $a$ has a large null space
consisting of gradients $\nabla\varphi$ of scalar functions $\varphi$.
This null space is not of interest in our application, and we would
like our numerical method to focus on the positive eigenvalues.
On the other hand, we not only have to solve one eigenvalue problem,
but a large number of eigenvalue problems for varying values of the
Bloch parameter $k$:
in order to find band gaps, we have to sample the entire rectangle
$[-\pi/a,\pi/a]\times[-\pi/b,\pi/b]$ of possible Bloch parameters at a
sufficiently fine resolution.

The first challenge can be met by using a discrete Helmholtz
decomposition \cite{HI98}:
if a bilinear edge element function is a gradient $\nabla\varphi$,
the corresponding potential $\varphi$ is a bilinear \emph{nodal}
function, i.e., it can be represented by the standard $Q_1$ nodal
basis.
We still have to address the question of boundary conditions:
if $\nabla\varphi\in\Hbloch$ holds, what are the appropriate
boundary conditions for $\varphi\in H^1(\Omega)$?
Our answer to this question is given in
\cref{le:potential_boundary_conditions}.

For the second challenge, we combine an extrapolation technique
based on the Rayleigh quotient with a preconditioned block inverse
iteration.
This approach allows us to compute the smallest non-zero eigenvalues
and a corresponding biorthogonal basis of eigenvectors for most
Bloch parameters $k$ with only a few iteration steps.

This text is organized as follows:
the following \cref{se:bloch_boundary_conditions} investigates
the influence of Bloch boundary conditions on the variational formulation
(cf. \cref{le:partial_integration}) and the Helmholtz decomposition
(cf. \cref{le:potential_boundary_conditions}).
Once the variational formulation is at our disposal, we consider in
\cref{se:discretization} the discretization with N\'ed\'elec's bilinear basis
functions of lowest order adjusted to handle the Bloch boundary conditions.
The discretization yields a generalized matrix eigenvalue problem that
we choose to solve with the preconditioned block inverse iteration
described in \cref{se:preconditioned_block_inverse_iteration} with
suitable modifications needed to handle the null space.
In \cref{se:geometric_multigrid_method} we describe the geometric multigrid
methods used in our implementation to provide a preconditioner for the
eigenvalue iteration and to remove the null space from the iteration vectors.
Since we have to solve a large number of eigenvalue problems in order
to cover the parameter domain, we employ a simple extrapolation technique
described in \cref{se:extrapolation} to obtain good initial values for
the eigenvalue iteration.
Our experiments indicate that the convergence of the preconditioned
inverse iteration can be improved significantly by computing a few
more eigenvectors than strictly required, and \cref{se:throwaway_eigenvectors}
describes this approach.
The final \cref{se:numerical_experiments} is devoted to numerical
experiments that indicate that our method performs as expected.

\section{Bloch boundary conditions}
\label{se:bloch_boundary_conditions}

The Bloch boundary conditions \cref{eq:bloch_boundary} have
a significant impact on the properties of the eigenvalue problem.
On the one hand, we have to verify that the variational formulation
\cref{eq:variational} is equivalent with the original problem
\cref{eq:maxwell}, particularly that no additional boundary
terms appear.
On the other hand, efficient numerical methods for Maxwell-type
problems rely on a Helmholtz decomposition, i.e., the decomposition
of $u\in H(\scurl,\Omega)$ into a gradient $u_0 := \nabla\varphi$
and a divergence-free function $u_1$.
If we impose Bloch boundary conditions for $u$, we have to investigate
what boundary conditions are appropriate for the potential $\varphi$ of
the Helmholtz decomposition.

We first consider how the Bloch boundary conditions influence
partial integration.
For this, we need a scalar counterpart of the Bloch boundary
conditions \cref{eq:bloch_boundary}:
For a function $\psi\in C^1(\bbbr^2)$, we consider the conditions
\begin{subequations}\label{eq:bloch_scalar}
\begin{align}
  \psi(x_1,b) &= \exp(\iota k_2 b)\, \psi(x_1,0) &
  &\text{ for all } x_1\in[0,a],\label{eq:bloch_scalar1}\\
  \psi(a,x_2) &= \exp(\iota k_1 a)\, \psi(0,x_2) &
  &\text{ for all } x_2\in[0,b].\label{eq:bloch_scalar2}
\end{align}
\end{subequations}
For a function $v$ with the boundary conditions \cref{eq:bloch_boundary}
and a function $\psi$ with the scalar boundary conditions
\cref{eq:bloch_scalar}, we can perform partial integration without
introducing additional boundary terms.

%
%
\begin{lemma}[Partial integration]
\label{le:partial_integration}
Let $\psi\in C^1(\Omega)$ satisfy the scalar Bloch boundary conditions
\cref{eq:bloch_scalar}, let $v\in C^1(\Omega,\bbbr^2)$ satisfy the
vector Bloch boundary conditions \cref{eq:bloch_boundary}.
We have
\begin{equation*}
  \int_\Omega \langle v(x), \vcurl \psi(x) \rangle \,dx
  = \int_\Omega \scurl \bar v(x) \psi(x) \,dx.
\end{equation*}
\end{lemma}
\begin{proof}
Using Gauss's theorem, we find
\begin{align*}
  \int_\Omega \langle v(x), \vcurl \psi(x) \rangle \,dx
  &= \int_\Omega (\bar v_2(x) \partial_1 \psi(x)
                 - \bar v_1(x) \partial_2 \psi(x)) \,dx\\
  &= -\int_\Omega (\partial_1 \bar v_2(x) \psi(x)
                  - \partial_2 \bar v_1(x) \psi(x)) \,dx\\
  &\quad + \int_{\partial\Omega} (n_1(x) \bar v_2(x) \psi(x)
                              - n_2(x) \bar v_1(x) \psi(x))\,dx\\
  &= \int_\Omega \scurl \bar v(x) \psi(x) \,dx
   + \int_{\partial\Omega} \langle v(x), t(x) \rangle \psi(x) \,dx,
\end{align*}
where $n:\partial\Omega\to\bbbr^2$ is the unit outer normal vector
and
\begin{align*}
  t(x) &:= \begin{pmatrix}
             -n_2(x)\\ n_1(x)
           \end{pmatrix} &
  &\text{ for all } x\in\partial\Omega
\end{align*}
is the counter-clockwise unit tangential vector.
Using the boundary conditions \cref{eq:bloch_boundary} and
\cref{eq:bloch_scalar}, we find
\begin{align*}
  \int_{\partial\Omega}
  \langle v(x), t(x) \rangle \psi(x) \,dx
  &= \int_0^a \bar v_1(x_1,0) \psi(x_1,0) \,dx_1
   - \int_0^a \bar v_1(x_1,b) \psi(x_1,b) \,dx_1\\
  &\quad + \int_0^b \bar v_2(a,x_2) \psi(a,x_2) \,dx_2
   - \int_0^b \bar v_2(0,x_2) \psi(0,x_2) \,dx_2\\
  &= \int_0^a \bar v_1(x_1,0) \psi(x_1,0) \,dx_1\\
  &\quad - \int_0^a \exp(-ik_2 b) \bar v_1(x_1,0)
                   \exp(\iota k_2 b) \psi(x_1,0) \,dx\\
  &\quad + \int_0^b \exp(-ik_1 a) \bar v_2(0,x_2)
                   \exp(\iota k_1 a) \psi(0,x_2) \,dx_2\\
  &\quad - \int_0^b \bar v_2(0,x_2) \psi(0,x_2) \,dx_2 = 0.
\end{align*}
\end{proof}
In order to derive the variational formulation \cref{eq:variational},
we have to apply this identity to $\psi = \scurl u$, where
$u$ is the solution of the partial differential equation.
Using Bloch's theorem again, we find a periodic function $\hat u$
such that
\begin{align*}
  u(x) &= \exp(\iota \langle k, x \rangle) \hat u(x) &
  &\text{ for all } x\in\bbbr^2.
\end{align*}
Using the chain rule yields
\begin{align*}
  \psi(x) &= \scurl u(x)
   = \partial_2 u_1(x) - \partial_1 u_2(x)\\
  &= i k_2 \exp(\iota \langle k, x \rangle) \hat u_1(x)
   + \exp(\iota \langle k, x \rangle) \partial_2 \hat u_1(x)\\
  &\quad + i k_1 \exp(\iota \langle k, x \rangle) \hat u_2(x)
   + \exp(\iota \langle k, x \rangle) \partial_1 \hat u_2(x)\\
  &= \exp(\iota \langle k, x \rangle) \bigl(
      i k_2 \hat u_1(x) - i k_1 \hat u_2(x)
      + \scurl \hat u(x) \bigr)\quad\text{ for all } x\in\bbbr^2.
\end{align*}
Using this equation, we can verify that the Bloch boundary
conditions \cref{eq:bloch_scalar} hold:
using the periodicity of $\hat u$, we find
\begin{align*}
  \psi(x_1,b)
  &= \exp\bigl(\iota (k_1 x_1 + k_2 b)\bigr)
     \bigl( i k_2 \hat u_1(x_1,b) - i k_1 \hat u_2(x_1,b)
            + \scurl \hat u(x_1,b) \bigr)\\
  &= \exp(\iota k_2 b) \exp(\iota k_1 x_1)
     \bigl( i k_2 \hat u_1(x_1,0) - i k_1 \hat u_2(x_1,0)
            + \scurl \hat u(x_1,0) \bigr)\\
  &= \exp(\iota k_2 b)\, \psi(x_1,0)
      \quad\text{ for all } x_1\in[0,a],\\
  \psi(a,x_2)
  &= \exp\bigl(\iota (k_1 a + k_2 x_2)\bigr)
     \bigl( i k_2 \hat u_1(a,x_2) - i k_1 \hat u_2(a,x_2)
            + \scurl \hat u(a,x_2) \bigr)\\
  &= \exp(\iota k_1 a) \exp(\iota k_2 x_2)
     \bigl( i k_2 \hat u_1(0,x_2) - i k_1 \hat u_2(0,x_2)
            + \scurl \hat u(0,x_2) \bigr)\\
  &= \exp(\iota k_1 a)\, \psi(0,x_2)
      \quad\text{ for all } x_2\in[0,b].
\end{align*}
Since the permittivity function $\epsilon$ is periodic, the
boundary conditions \cref{eq:bloch_scalar} also hold
for $\psi = \tfrac{1}{\epsilon} \scurl u$, therefore the boundary terms
appearing in the partial integration cancel and we obtain
\begin{align}\label{eq:partial_integration}
  \int_\Omega \langle v(x), \vcurl
                    \frac{1}{\epsilon(x)} \scurl u(x) \rangle \,dx
  &= \int_\Omega \frac{\scurl \bar v(x) \scurl u(x)}{\epsilon(x)} \,dx
    \quad\text{ for all } v\in C^1(\Omega,\bbbr^2).
\end{align}
Since $C^1(\Omega,\bbbr^2)$, equipped with the Bloch boundary
conditions \cref{eq:bloch_boundary}, is a dense subspace of $\Hbloch$,
we find that every solution of \cref{eq:maxwell} is also a
solution of the variational problem \cref{eq:variational}.

Now we can consider the second issue with Bloch boundary
conditions:
how do they influence the Helmholtz decomposition?

%
%
\begin{lemma}[Potentials]
\label{le:potential_boundary_conditions}
Let $\varphi\in C^1(\Omega)$ with $\nabla\varphi\in\Hbloch$.
If $k\neq 0$, there is a constant $m\in\bbbc$ such that
$\hat\varphi:=\varphi-m$ satisfies the scalar Bloch boundary conditions
\cref{eq:bloch_scalar}.

Otherwise, i.e., in the special case $k=0$, there is a linear polynomial
$\mu$ such that $\hat\varphi:=\varphi-\mu$ satisfies these conditions.
\end{lemma}
\begin{proof}
Let $\alpha := \exp(\iota k_1 a)$ and $\beta := \exp(\iota k_2 b)$.
By the fundamental theorem of calculus and \cref{eq:bloch_boundary},
we have
\begin{subequations}
\begin{align}
  \varphi(x_1,b) - \varphi(0,b)
  &= \int_0^{x_1} \partial_1\varphi(t,b) \,dt
   = \int_0^{x_1} u_1(t,b) \,dt
   = \beta \int_0^{x_1} u_1(t,0) \,dt\notag\\
  &= \beta \int_0^{x_1} \partial_1\varphi(t,0) \,dt
   = \beta \bigl( \varphi(x_1,0) - \varphi(0,0) \bigr)
     \label{eq:fundamental_beta}
\end{align}
for all $x_1\in[0,a]$.
In order to satisfy \cref{eq:bloch_scalar1}, we have to ensure
\begin{equation*}
  \beta\, \varphi(x_1,0)
  = \varphi(x_1,b)
  = \beta \bigl( \varphi(x_1,0) - \varphi(0,0) \bigr) + \varphi(0,b)
\end{equation*}
for all $x_1\in[0,a]$, i.e., $\beta \varphi(0,0) = \varphi(0,b)$.
By the same reasoning, we find
\begin{align}
  \varphi(a,x_2) - \varphi(a,0)
  &= \int_0^{x_2} \partial_2\varphi(a,s) \,ds
   = \int_0^{x_2} u_2(a,s) \,ds
   = \alpha \int_0^{x_2} u_2(0,s) \,ds\notag\\
  &= \alpha \int_0^{x_2} \partial_2\varphi(0,s) \,ds
   = \alpha \bigl(\varphi(0,x_2) - \varphi(0,0)\bigr)
     \label{eq:fundamental_alpha}
\end{align}
\end{subequations}
for all $x_2\in[0,b]$, and satisfying \cref{eq:bloch_scalar2}
is equivalent with
\begin{equation*}
  \alpha \varphi(0,x_2)
  = \varphi(a,x_2)
  = \alpha \bigl(\varphi(0,x_2) - \varphi(0,0)\bigr) + \varphi(a,0)
\end{equation*}
for all $x_2\in[0,b]$, i.e., $\alpha \varphi(0,0) = \varphi(a,0)$.
Our task is now to see that we can satisfy both
$\alpha \varphi(0,0) = \varphi(a,0)$ and
$\beta \varphi(0,0) = \varphi(0,b)$ simultaneously.
Substituting $x_1=a$ in \cref{eq:fundamental_beta} and $x_2=b$
in \cref{eq:fundamental_alpha} implies
\begin{subequations}\label{eq:alpha_beta_compatibility}
\begin{align}
  \varphi(a,b) &= \beta\bigl( \varphi(a,0) - \varphi(0,0) \bigr)
                 + \varphi(0,b),\\
  \varphi(a,b) &= \alpha\bigl( \varphi(0,b) - \varphi(0,0) \bigr)
                 + \varphi(a,0).
\end{align}
\end{subequations}
We distinguish three cases: $\alpha\neq 1$, $\beta\neq 1$, and
$\alpha=\beta=1$.

\emph{Case 1:}
We assume $\alpha\neq 1$ and let
\begin{align*}
  m &:= \frac{\varphi(a,0) - \alpha \varphi(0,0)}{1-\alpha}, &
  \hat\varphi &:= \varphi-m.
\end{align*}
This choice implies
\begin{equation*}
  \alpha\hat\varphi(0,0)
  = \alpha\varphi(0,0) - m + (1-\alpha) m
  = \varphi(a,0) - m
  = \hat\varphi(a,0).
\end{equation*}
The equations \cref{eq:alpha_beta_compatibility} also hold for
$\hat\varphi$ and we obtain
\begin{align*}
  \alpha \hat\varphi(0,b)
  &= \hat\varphi(a,b)
   = \beta \bigl( \alpha\hat\varphi(0,0) - \hat\varphi(0,0) \bigr)
       + \hat\varphi(0,b),\\
  (\alpha-1) \hat\varphi(0,b)
  &= \beta (\alpha-1) \hat\varphi(0,0),
\end{align*}
and this implies $\beta\hat\varphi(0,0) = \hat\varphi(0,b)$.

\emph{Case 2:}
We assume $\beta\neq 1$ and let
\begin{align*}
  m &:= \frac{\varphi(b,0) - \beta \varphi(0,0)}{1-\beta}, &
  \hat\varphi &:= \varphi+m.
\end{align*}
This choice implies
\begin{equation*}
  \beta\hat\varphi(0,0)
  = \beta\varphi(0,0) - m + (1-\beta) m
  = \varphi(0,b) - m
  = \hat\varphi(0,b).
\end{equation*}
Again we use that the equations \cref{eq:alpha_beta_compatibility}
also hold for $\hat\varphi$ to find
\begin{align*}
  \beta \hat\varphi(a,0)
  &= \hat\varphi(a,b)
   = \alpha \bigl( \beta\hat\varphi(0,0) - \hat\varphi(0,0) \bigr)
       + \hat\varphi(a,0),\\
  (\beta-1) \hat\varphi(a,0)
  &= \alpha (\beta-1) \hat\varphi(0,0),
\end{align*}
and this implies $\alpha\hat\varphi(0,0) = \hat\varphi(a,0)$.

\emph{Case 3:}
We assume $\alpha=\beta=1$, i.e., $k=0$, and define
\begin{align*}
  \mu(x) &:= \frac{\varphi(a,0) - \varphi(0,0)}{2 a} (2 x_1 - a)
           + \frac{\varphi(0,b) - \varphi(0,0)}{2 b} (2 x_2 - b) &
  &\text{ for all } x\in\bbbr^2.
\end{align*}
Let $\hat\varphi := \varphi - \mu$.
We find
\begin{align*}
  \hat\varphi(a,0)
  &= \varphi(a,0) - \mu(a,0)
   = \varphi(a,0) - \frac{\varphi(a,0) - \varphi(0,0)}{2}
     + \frac{\varphi(0,b) - \varphi(0,0)}{2}\\
  &= \frac{\varphi(a,0) + \varphi(0,0)}{2}
     + \frac{\varphi(0,b) - \varphi(0,0)}{2}\\
  &= \varphi(0,0) + \frac{\varphi(a,0) - \varphi(0,0)}{2}
     + \frac{\varphi(0,b) - \varphi(0,0)}{2}\\
  &= \varphi(0,0) - \mu(0,0) = \hat\varphi(0,0),\\
  \hat\varphi(0,b)
  &= \varphi(0,b) - \mu(0,b)
   = \varphi(0,b) + \frac{\varphi(a,0) - \varphi(0,0)}{2}
     - \frac{\varphi(0,b) - \varphi(0,0)}{2}\\
  &= \frac{\varphi(a,0) - \varphi(0,0)}{2}
     + \frac{\varphi(0,b) + \varphi(0,0)}{2}\\
  &= \varphi(0,0) + \frac{\varphi(a,0) - \varphi(0,0)}{2}
     + \frac{\varphi(0,b) - \varphi(0,0)}{2}\\
  &= \varphi(0,0) - \mu(0,0) = \hat\varphi(0,0).
\end{align*}
Now we can use \cref{eq:fundamental_beta} with $\beta=1$ to find
\begin{align*}
  \hat\varphi(x_1,b) - \hat\varphi(0,b)
  &= \varphi(x_1,b) - \varphi(0,b)
     - \frac{\varphi(a,0) - \varphi(0,0)}{a} x_1\\
  &= \varphi(x_1,0) - \varphi(0,0)
     - \frac{\varphi(a,0) - \varphi(0,0)}{a} x_1\\
  &= \hat\varphi(x_1,0) - \hat\varphi(0,0)
   = \hat\varphi(x_1,0) - \hat\varphi(0,b),\\
\intertext{for all $x_1\in[0,a]$ and similarly
\cref{eq:fundamental_alpha} with $\alpha=1$ to get}
  \hat\varphi(a,x_2) - \hat\varphi(a,0)
  &= \varphi(a,x_2) - \varphi(a,0)
     - \frac{\varphi(0,b) - \varphi(0,0)}{b} x_2\\
  &= \varphi(0,x_2) - \varphi(0,0)
     - \frac{\varphi(0,b) - \varphi(0,0)}{b} x_2\\
  &= \hat\varphi(0,x_2) - \hat\varphi(0,0)
   = \hat\varphi(0,x_2) - \hat\varphi(a,0)
\end{align*}
for all $x_2\in[0,b]$, i.e., $\hat\varphi$ satisfies the scalar
Bloch boundary conditions \cref{eq:bloch_scalar}.
\end{proof}

We can conclude that Bloch boundary conditions can serve a similar
purpose as the widely used Dirichlet boundary conditions:
when deriving the variational formulation, they eliminate boundary
terms appearing during partial integration, and when applying
the Helmholtz decomposition, they ensure uniqueness of the
gradient if $k\neq 0$ and uniqueness up to an explicitly known
two-dimensional subspace, i.e., the gradients of linear polynomials,
in the special case $k=0$.

\section{Discretization}
\label{se:discretization}

We discretize the variational eigenvalue problem \cref{eq:variational}
on a regular rectangular mesh $\mathcal{T}$ using N\'ed\'elec's
bilinear edge elements \cite{NE80}:
we choose $n,m\in\bbbn$ and split the domain $\Omega$ into
$n\times m$ rectangular mesh cells of width $h_1 := a/n$ and height
$h_2 := b/m$ given by
\begin{align*}
  \Omega_i &:= [(i_1-1) h_1, i_1 h_1]\times[(i_2-1) h_2, i_2 h_2] &
  &\text{ for all } i\in\Idx:=[1:n]\times[1:m].
\end{align*}
We modify the standard definition of N\'ed\'elec's edge element
basis functions to include the Bloch boundary conditions
\cref{eq:bloch_boundary}:
the support of basis functions on the top horizontal edge wraps over to
the lower edge, and the value on the top edge is equal to the value on the
bottom edge multiplied by $\exp(\iota k_2 b)$.
\begin{align*}
  b_{x,i}(x) &:= \begin{cases}
               \frac{1}{h_1} (x_2/h_2 - i_2 + 1, 0)
               &\text{ if } x\in\Omega_i,\ i_2<m,\\
               \frac{1}{h_1} (i_2 + 1 - x_2/h_2, 0)
               &\text{ if } x\in\Omega_{i_1,i_2+1},\ i_2<m\\
               \frac{\exp(\iota k_2 b)}{h_1} (x_2/h_2 - i_2 + 1,0)
               &\text{ if } x\in\Omega_i,\ i_2=m,\\
               \frac{1}{h_1} (1 - x_2/h_2, 0)
               &\text{ if } x\in\Omega_{i_1,1},\ i_2=m\\
               (0,0) &\text{ otherwise}
             \end{cases}
\intertext{for all $i\in\Idx$, $x\in\Omega$.
For basis functions corresponding to vertical edges, we incorporate
the Bloch boundary conditions by setting the value on the right vertical
edge by multiplying the values on the left vertical edge by
$\exp(\iota k_1 a)$.}
  b_{y,i}(x) &:= \begin{cases}
               \frac{1}{h_2} (0, x_1/h_1 - i_1 + 1, 0)
               &\text{ if } x\in\Omega_i,\ i_1<n,\\
               \frac{1}{h_2} (0, i_1 + 1 - x_1/h_1)
               &\text{ if } x\in\Omega_{i_1+1,i_2},\ i_1<n\\
               \frac{\exp(\iota k_1 a)}{h_2} (0, x_1/h_1 - i_1 + 1)
               &\text{ if } x\in\Omega_i,\ i_1=n,\\
               \frac{1}{h_2} (0, 1 - x_1/h_1)
               &\text{ if } x\in\Omega_{1,i_2},\ i_1=n\\
               (0,0) &\text{ otherwise}
             \end{cases}
\end{align*}
for all $i\in\Idx$, $x\in\Omega$.
The N\'ed\'elec space with Bloch boundary conditions is given by
\begin{equation*}
  V_h := \mathop{\operatorname{span}}\{
           b_{x,i},\ b_{y,i}\ :\ i\in\Idx \}.
\end{equation*}
In order to handle the null space of the $\scurl$ operator, we
also need the space of scalar bilinear functions with scalar Bloch
boundary conditions \cref{eq:bloch_scalar} on the same grid.
The basis functions are defined using the one-dimensional hat functions
\begin{align*}
  \phi_{x,i}(x)
  &:= \begin{cases}
    x/h_1-i+1 &\text{ if } x\in[(i-1)h_1,ih_1],\ i<n,\\
    i+1-x/h_1 &\text{ if } x\in[ih_1,(i+1)h_1],\ i<n,\\
    \exp(\iota k_1 a) (x/h_1-i+1) &\text{ if } x\in[(i-1)h_1,ih_1],\ i=n,\\
    1-x/h_1 &\text{ if } x\in[0,h_1],\ i=n,
  \end{cases}\\
  \phi_{y,j}(y)
  &:= \begin{cases}
    y/h_2-j+1 &\text{ if } y\in[(j-1)h_2,jh_2],\ j<m,\\
    j+1-y/h_2 &\text{ if } y\in[jh_2,(j+1)h_2],\ j<m,\\
    \exp(\iota k_2 b) (y/h_2-j+1) &\text{ if } y\in[(j-1)h_2,jh_2],\ j=m,\\
    1-y/h_2 &\text{ if } y\in[0,h_2],\ j=m
  \end{cases}
\end{align*}
defined for $x\in[0,a]$, $y\in[0,b]$, $i\in[1:n]$ and $j\in[1:m]$ and
taking the one-dimensional counterparts of the Bloch boundary conditions
\cref{eq:bloch_scalar} into account.
The bilinear nodal basis functions are defined by the tensor products
\begin{align*}
  \varphi_i(x) &:= \phi_{x,i_1}(x_1) \phi_{y,i_2}(x_2) &
  &\text{ for all } i\in\Idx,\ x\in\Omega
\end{align*}
and satisfy \cref{eq:bloch_scalar} by definition.
The nodal space with Bloch boundary conditions \cref{eq:bloch_scalar} is
given by
\begin{equation*}
  W_h := \mathop{\operatorname{span}}\{ \varphi_i\ :\ i\in\Idx \}.
\end{equation*}
We can see that
\begin{align}\label{eq:potential_discrete}
  \nabla\varphi_i
  &= b_{x,i} - b_{x,(i_1,i_2+1)}
     + b_{y,i} - b_{y,(i_1+1,i_2)} &
  &\text{ for all } i\in\Idx,
\end{align}
i.e., gradients of nodal basis functions can be expressed exactly and
explicitly in terms of four edge basis functions.
Using the Helmholtz decomposition, the well-known properties of N\'ed\'elec
elements and \cref{le:potential_boundary_conditions}, we can prove in the
case $k\neq 0$ that for every $u_h\in V_h$ with $\scurl u_h=0$, there is a
$\varphi_h\in W_h$ with $u_h = \nabla\varphi_h$.
This property allows us to eliminate the null space of the bilinear
form $a(\cdot,\cdot)$ in our algorithm.

In the special case $k=0$, we can still eliminate the null space up to
a two-dimensional remainder that is explicitly known.

The matrices
$A_{xx},A_{xy},A_{yy},M_{xx},M_{xy},M_{yy},G\in\bbbc^{\Idx\times\Idx}$
resulting from a standard Ga\-ler\-kin discretization are given by
\begin{gather*}
  a_{xx,ij} := a(b_{x,i}, b_{x,j}),\quad
  a_{xy,ij} := a(b_{x,i}, b_{y,j}),\quad
  a_{yy,ij} := a(b_{y,i}, b_{y,j}),\\
  m_{xx,ij} := m(b_{x,i}, b_{x,j}),\quad
  m_{xy,ij} := m(b_{x,i}, b_{y,j}),\quad
  m_{yy,ij} := m(b_{y,i}, b_{y,j}),\\
  g_{ij} := m(\nabla\varphi_i, \nabla\varphi_j)
\end{gather*}
for all $i,j\in\Idx$.
In addition, we introduce the \emph{lifting matrices}
$L_x,L_y\in\bbbc^{\Idx\times\Idx}$
such that
\begin{align*}
  \nabla\varphi_j &= \sum_{i\in\Idx}
    \begin{pmatrix}
      \ell_{x,ij} b_{x,i}\\
      \ell_{y,ij} b_{y,j}
    \end{pmatrix} &
  &\text{ for all } j\in\Idx.
\end{align*}
These matrices exist due to \cref{eq:potential_discrete}.
To ease notation, we introduce the block matrices
\begin{gather*}
  A := \begin{pmatrix}
    A_{xx} & A_{xy}^*\\
    A_{xy} & A_{yy}
  \end{pmatrix}\in\bbbc^{\Jdx\times\Jdx},\quad
  M := \begin{pmatrix}
    M_{xx} & M_{xy}^*\\
    M_{xy} & M_{yy}
  \end{pmatrix}\in\bbbc^{\Jdx\times\Jdx},\quad
  L := \begin{pmatrix}
    L_x\\
    L_y
  \end{pmatrix}\in\bbbc^{\Jdx\times\Idx}
\end{gather*}
with $\Jdx=\Idx\times\{1,2\}$ and $|\Jdx|=2nm$ in
order to obtain the desired form \cref{eq:matrix} of the discrete
eigenvalue problem.

\section{Preconditioned block inverse iteration}
\label{se:preconditioned_block_inverse_iteration}

We are interested in computing a few of the smallest non-zero
eigenvalues and the corresponding eigenvectors.
We base our approach on the preconditioned inverse iteration (PINVIT)
\cite{SA58,BRPAKN96,KNNE03}:
to find an eigenvector of \cref{eq:matrix}, we consider the sequence
$(e^{(m)})_{m=0}^\infty$ in $\bbbc^{\Jdx}\setminus\{0\}$ defined by
\begin{align*}
  e^{(m+1)} &= e^{(m)} - B (A e^{(m)} - \lambda_m M e^{(m)}) &
  &\text{ for all } m\in\bbbn_0,
\end{align*}
where $B$ is an approximation of $A^{-1}$ and
\begin{align*}
  \lambda_m &:= \frac{\langle e^{(m)}, A e^{(m)} \rangle}
                     {\langle e^{(m)}, M e^{(m)} \rangle} &
  &\text{ for all } m\in\bbbn_0
\end{align*}
is the generalized Rayleigh quotient.
We can see that any solution of the generalized eigenvalue problem
\cref{eq:matrix} is a fixed point of this iteration and that in
the case $B=A^{-1}$ it is identical (up to scaling) to the standard
inverse iteration.

For our application, we have to make a few adjustments:
every gradient of a scalar potential is in the null space of
the $\scurl$ operator, and we are not interested in the zero eigenvalue
of inifinite multiplicity.
Fortunately, N\'ed\'elec edge elements \cite{NE80} offer an elegant
solution: on the one hand, they avoid ``spurious modes'' that trouble
standard nodal finite element methods, on the other hand, all
elements of the discrete null space are gradients of scalar
piecewise polynomial functions on the same mesh.

Using Lemma~\ref{le:potential_boundary_conditions}, we can eliminate
the elements of the null space:
Given a function $u\in\Hbloch$, for $k\neq 0$ we can find a
potential $\varphi\in H^1(\Omega)$ that
satisfies the scalar Bloch boundary conditions \cref{eq:bloch_scalar}
by solving
\begin{align*}
  m(\nabla v,\nabla\varphi)
  &= m(\nabla v, u) &
  &\text{ for all } v\in H^1(\Omega)
    \text{ with } \cref{eq:bloch_scalar}.
\end{align*}
Obviously, this $\varphi$ will satisfy
\begin{align*}
  m(\nabla v, u - \nabla\varphi) &= 0 &
  &\text{ for all } v\in H^1(\Omega)
    \text{ with } \cref{eq:bloch_scalar},
\end{align*}
i.e., $u-\nabla\varphi$ will be perpendicular on all gradients and
therefore also perpendicular on the null space of the $\scurl$ operator.
Due to the special properties of N\'ed\'elec elements
\cref{eq:potential_discrete}, the same holds for the discrete setting,
i.e., we have $A L = 0$, and by solving the equation
\begin{equation*}
  L^* M L \varphi_h = L^* M u_h
\end{equation*}
and computing $u_h' := u_h - L \varphi_h$, we can ensure that the vector
$u_h'$ is perpendicular on the null space of $A$, i.e., that the zero
eigenvalue is eliminated.
In the special case $k=0$, we can either eliminate the remaining
two-dimensional subspace explicitly or simply disregard the zero eigenvalue.

Since we are typically interested in computing not just one, but
several eigenvectors corresponding to the smallest non-zero eigenvalues,
we employ a block method:
The iterates are matrices $E^{(m)}\in\bbbc^{\Jdx\times p}$, where $p\in\bbbn$
denotes the number of eigenvectors computed simultaneously.
In order to avoid all columns converging to the same eigenspace, we
ensure that the columns are an orthonormal basis with respect to the
mass matrix $M$, i.e., $(E^{(m)})^* M E^{(m)} = I$ has to hold for
all $m\in\bbbn_0$.
One step of the preconditioned inverse iteration takes the form
\begin{align*}
  \widehat{E}^{(m+1)}
  &= E^{(m)} - B ( A E^{(m)} - M E^{(m)} \Lambda_m ), &
  \Lambda_m &:= (E^{(m)})^* A E^{(m)},
\end{align*}
where $(E^{(m)})^* M E^{(m)} = I$ allows us to simplify the generalized
Rayleigh quotient $\Lambda_m$.
Unfortunately, $\widehat{E}^{(m+1)}$ will usually not satisfy our
$M$-orthonormality assumption, so we have to orthonormalize it.
For the sake of numerical stability, we use a generalized Householder
factorization of $\widehat{E}^{(m+1)}$:
We start with a prescribed $M$-orthonormal basis $P\in\bbbc^{\Jdx\times p}$,
i.e., $P^* M P = I$, in our case a suitable choice of canonical $M$-unit vectors
with disjoint supports, and find Householder vectors
$v_1,\ldots,v_p\in\bbbc^\Jdx\setminus\{0\}$ with generalized reflections
\begin{equation*}
  Q_i := I - 2 v_i \frac{v_i^* M}{v_i^* M v_i}
\end{equation*}
such that $Q_p Q_{p-1} \ldots Q_2 Q_1 \widehat{E}^{(m+1)} = P R$, where
$R\in\bbbc^{p\times p}$ is a right upper triangular matrix.
These generalized reflections satisfy $Q_i^2=I$, and
\begin{align*}
  \langle x, M Q_i y \rangle
  &= \langle x, M y \rangle
   - 2 \langle x, M v_i \rangle
     \frac{\langle v_i, M y \rangle}{\langle v_i, M v_i \rangle}
   = \langle x, M y \rangle
   - 2 \langle \frac{\overline{\langle x, M v_i \rangle}}
                    {\langle v_i, M v_i \rangle} v_i, M y \rangle\\
  &= \langle x - 2 v_i \frac{\langle v_i, M x \rangle}
                            {\langle v_i, M v_i \rangle}, M y \rangle
   = \langle Q_i x, M y \rangle
   \qquad\text{ for all } x,y\in\bbbc^\Jdx
\end{align*}
shows that they are also $M$-selfadjoint.

If $R$ is invertible, we obtain
\begin{align*}
  Q_p Q_{p-1} \ldots Q_2 Q_1 \widehat{E}^{(m+1)} R^{-1} &= P, &
  \widehat{E}^{(m+1)} R^{-1} = Q_1 Q_2 \ldots Q_{p-1} Q_p P,
\end{align*}
and the right-hand side $E^{(m+1)} := Q_1 Q_2 \ldots Q_{p-1} Q_p P$ is
$M$-orthonormal by construction, so we can use it for the next iteration step.

If $R$ is not invertible, $\widehat{E}^{(m+1)}$ cannot have full rank, i.e.,
we have started the iteration with an unsuitable initial guess.
Fortunately, in this case our choice of $E^{(m+1)}$ is still $M$-orthonormal
and its range will contain the range of $\widehat{E}^{(m+1)}$, so the
algorithm corrects the problem by implicitly extending the basis and
guarantees that we always have an $M$-orthonormal basis at our disposal.

A simple version of the resulting modified block preconditioned inverse
iteration takes the following form:
\begin{quotation}
\medskip
\begin{tabbing}
  Find $E^{(0)}\in\bbbc^{\Jdx\times p}$ with $(E^{(0)})^* M E^{(0)} = I$.\\
  $\Lambda_0 \gets (E^{(0)})^* A E^{(0)}$\\
  $m \gets 0$\\
  \texttt{while} $\|A E^{(m)} - E^{(m)} \Lambda_m\|$ too large\\
  \quad\= $\widehat E^{(m+1)} \gets
       E^{(m)} - B (A E^{(m)} - M E^{(m)} \Lambda_m)$\\
  \> Solve $L^* M L \Phi_{m+1} = L^* M \widehat{E}^{(m+1)}$\\
  \> $\widehat{E}^{(m+1)} \gets \widehat{E}^{(m+1)} - L \Phi_{m+1}$\\
  \> Generalized Householder factorization
       $E^{(m+1)} R^{(m+1)} = \widehat{E}^{(m+1)}$\\
  \> \quad with $(E^{(m+1)})^* M E^{(m+1)} = I$\\
  \> $\Lambda_{m+1} \gets (E^{(m+1)})^* A E^{(m+1)}$\\
  \> $m \gets m+1$\\
  \texttt{end}
\end{tabbing}
\medskip
\end{quotation}
We perform one step of the preconditioned inverse iteration for every
column of $E^{(m)}$, eliminate the null space by ensuring that all
iteration vectors are orthogonal on the space of gradients, and
turn the resulting vectors into an orthonormal basis.

We can speed up convergence considerably by computing the
\emph{Ritz vectors and values}, i.e., the $p$-dimensional Schur decomposition
\begin{align}\label{eq:ritz_eigenvalue}
  U_{m+1}^* \Lambda_{m+1} U_{m+1} &= D_{m+1}, &
  D_{m+1} &= \begin{pmatrix}
    \lambda_1 & & \\
    & \ddots & \\
    & & \lambda_p
  \end{pmatrix},\ \lambda_1\leq\lambda_2\leq \ldots\leq \lambda_p
\end{align}
with a unitary matrix $U_{m+1}\in\bbbc^{p\times p}$ and a real diagonal
matrix $D_{m+1}\in\bbbc^{p\times p}$ and replacing $E^{(m+1)}$ by
$E^{(m+1)} U_{m+1}$.
The latter matrix still has $M$-orthonormal columns, but these
columns are now approximations of the eigenvectors of $A$ and $M$,
while $\lambda_1,\ldots,\lambda_p$ are approximations of the corresponding
eigenvalues.

Since the $p$ smallest eigenvalues are the $p$ smallest local minima
of the Rayleigh quotient, we can improve the convergence speed by
looking for local minima not only in the range of $\widehat{E}^{(m+1)}$,
but in a larger subspace constructed by including the range of $E^{(m)}$
or even the range of $E^{(m-1)}$ if $m>1$.
In the first case, i.e., if we compute the generalized Householder
factorization
\begin{equation*}
  E^{(m+1)} R^{(m+1)}
  = \begin{pmatrix} \widehat{E}^{(m+1)} & E^{(m)} \end{pmatrix},
\end{equation*}
we arrive at the \emph{gradient method} for the minimization of the Rayleigh
quotient.
In the second case, i.e., if we compute the generalized Householder
factorization
\begin{equation*}
  E^{(m+1)} R^{(m+1)}
  = \begin{pmatrix}
      \widehat{E}^{(m+1)} & E^{(m)} & E^{(m-1)}
    \end{pmatrix}
\end{equation*}
if $m>1$, we get the \emph{locally optimal block preconditioned conjugate gradient
(LOBPCG)} method \cite{KN01} for the minimization task.
\begin{quotation}
\medskip
\begin{tabbing}
  Find $E^{(0)}\in\bbbc^{\Jdx\times p}$ with $(E^{(0)})^* M E^{(0)} = I$.\\
  $\Lambda_0 \gets (E^{(0)})^* A E^{(0)}$\\
  $m \gets 0$\\
  \texttt{while} $\|A E^{(m)} - E^{(m)} \Lambda_m\|$ too large\\
  \quad\= $\widehat E^{(m+1)} \gets
       E^{(m)} - B (A E^{(m)} - M E^{(m)} \Lambda_m)$\\
  \> Solve $L^* M L \Phi_{m+1} = L^* M \widehat{E}^{(m+1)}$\\
  \> $\widehat{E}^{(m+1)} \gets \widehat{E}^{(m+1)} - L \Phi_{m+1}$\\
  \> Householder factorization
       $E^{(m+1)} R^{(m+1)} = \begin{pmatrix}
          \widehat{E}^{(m+1)} & E^{(m)}
        \end{pmatrix}$\\
  \> \quad or $E^{(m+1)} R^{(m+1)} = \begin{pmatrix}
          \widehat{E}^{(m+1)} & E^{(m)} & E^{(m-1)}
        \end{pmatrix}$\\
  \> \quad with $(E^{(m+1)})^* M E^{(m+1)} = I$\\
  \> Solve \cref{eq:ritz_eigenvalue}, i.e.,
       $(E^{(m+1)})^* A E^{(m+1)} = U_{m+1} D_{m+1} U_{m+1}^*$\\
  \> $E^{(m+1)} \gets (E^{m+1} U_{m+1})|_{\Jdx\times p}$\\
  \> $\Lambda_{m+1} \gets D_{m+1}|_{p\times p}$\\
  \> $m \gets m+1$\\
  \texttt{end}
\end{tabbing}
\medskip
\end{quotation}
By construction, the columns of $E^{(m)}$ will be bi-orthogonal, i.e.,
orthonormal with respect to the $M$ inner product and orthogonal
with respect to the $A$ inner product.
If the range of $E^{(m)}$ is a good approximation of an invariant
subspace, the columns of $E^{(m)}$ are good approximations of eigenvectors
spanning this subspace.

\section{Geometric multigrid method}
\label{se:geometric_multigrid_method}

The preconditioned inverse iteration requires an efficient preconditioner
that approximates $A^{-1}$ sufficiently well.
In our case, $A$ is only positive semidefinite, so we replace it by
the positive definite matrix $A + \mu M$, where $\mu>0$ is a
regularization parameter.
This only shifts the eigenvalues by $\mu$ and does not change the
eigenvectors.

We employ a standard geometric multigrid method for Maxwell's equations
with suitable adjustments:
following \cite{ARFAWI00}, we use a block Gauss-Seidel smoother, where
each of the overlapping blocks corresponds to all edges connected to
a node of the grid, taking periodicity into account.
The four-dimensional linear systems corresponding to the individual
blocks are self-adjoint and positive definite and have one eigenvalue
that is considerably smaller than the others, which leads to a large
condition number and therefore poor numerical stability of the original
implementation.
The problematic eigenvalue corresponds to the gradient of the nodal
basis function of the current grid node, so we can employ an orthogonal
transformation to separate this eigenspace from its orthogonal complement.
Solving the resulting block-diagonal system using a standard Cholesky
factorization leads to a numerically stable algorithm.

The hierarchy of coarse grids is constructed by simple bisection.
This approach allows us to use the simple identical embedding as a
natural prolongation mapping the coarse grid into the next-finer grid,
and we can use the standard Galerkin approach to construct the corresponding
coarse-grid matrices:
The finest mesh has to be sufficiently fine to resolve the jumps in
the permittivity parameter $\epsilon$, so the corresponding mass and
stiffness matrices can be constructed by standard quadrature.
For a coarse mesh, we map trial and test basis functions to the next-finer
mesh using the natural embedding as a prolongation and evaluate the bilinear
form there.
In this way, the exact mass and stiffness matrices can be constructed
for all meshes in linear complexity.

In order to handle the null space, we have to solve linear systems
with the self-adjoint matrix $P := L^* M L$ corresponding to the discrete
Laplace operator with Bloch boundary conditions \cref{eq:bloch_scalar}.
If $k\neq 0$, $P$ is positiv definite, in the special case $k=0$ it
is positiv semidefinite with a two-dimensional null space spanned by
discretized linear polynomials.
The matrices $P$ for the entire mesh hierarchy can again be constructed by
the Galerkin approach, and we can use the corresponding standard multigrid
iteration to approximate the null-space projection.
Our experiments indicate that a few multigrid steps are sufficient to
stop the eigenvector approximation from converging to the null space,
we do not have to wait for the multigrid iteration to compute the
exact projection.

\section{Extrapolation}
\label{se:extrapolation}

Since the preconditioned inverse iteration is non-linear due to the
non-linear influence of the Rayleigh quotient $\Lambda_m$, it is
crucial to provide it with good initial guesses for the eigenvectors.

For the first Bloch parameter $k$ under consideration, we employ
a simple nested iteration:
on the coarsest mesh, the eigenvalue problem is solved by a direct
method.
Once the eigenvectors for the mesh level $\ell$ have been computed at
a sufficient accuracy, we use the prolongation to map them to the
next-finer grid level $\ell+1$ and use them as initial guesses for
the iteration on this level.

This procedure is only applied for the first Bloch parameter.

For all other Bloch parameters, we use an algorithm that is related
to extrapolation:
Assume that eigenvector bases $E_1,E_2,\ldots,E_e$, $e\in\bbbn$, have been
computed in previous steps for Bloch parameters ``close'' to the current
parameter $k$.
For standard polynomial extrapolation, we would have to construct polynomials
$p_i$ such that $p_i(E_1,E_2,\ldots,E_e)$ is a good approximation of the
$i$-th eigenvector.
Fortunately, we can again use the Courant-Fischer theorem to avoid this
task:
instead of constructing polynomials $p_i$ explicitly, we look for the
$p$ smallest non-zero minima of the Rayleigh quotient in the space
spanned by the ranges of $E_1,\ldots,E_e$.
The corresponding vectors form an $M$-orthonormal basis of a subspace
that serves as our initial guess for the eigenvectors for the current
Bloch parameter.

Due to the Courant-Fischer theorem, the vectors constructed in this
way are at least as good as the best possible polynomial approximation,
i.e., at least as good as the best extrapolation scheme.

Our experiments indicate that quadratic one-dimensional extrapolation
is sufficient to provide good initial guesses for the eigenvector
iteration.
Denoting the sampled Bloch parameters by
\begin{align*}
  k_{ij} &:= (\tfrac{\pi}{a} (2 \tfrac{i}{\kappa-1}-1),
             \tfrac{\pi}{b} (2 \tfrac{j}{\kappa-1}-1) &
  &\text{ for all } i,j\in[0:\kappa-1],
\end{align*}
we apply extrapolation as follows:
$k_{00}$ is computed directly.
$k_{10}$ is extrapolated ``horizontally'' using $k_{00}$.
$k_{20}$ is extrapolated ``horizontally'' using $k_{00}$ and $k_{10}$.
$k_{i0}$ is extrapolated ``horizontally'' using $k_{i-3,0}$, $k_{i-2,0}$,
and $k_{i-1,0}$ for all $i\in[3:\kappa-1]$.

$k_{01}$, $k_{11}$, and $k_{21}$ are extrapolated ``vertically'' using
 $k_{00}$, $k_{10}$, and $k_{20}$, respectively.
$k_{02}$, $k_{12}$, and $k_{22}$ are extrapolated ``vertically'' using
$k_{00}$ and $k_{01}$, $k_{10}$ and $k_{11}$, and $k_{20}$ and $k_{21}$,
respectively.
$k_{0j}$, $k_{1j}$, and $k_{2j}$ are extrapolated ``vertically'' using
three points for $j\geq 3$.
$k_{ij}$ is extrapolated ``horizontally'' using $k_{i-3,j}$,
$k_{i-2,j}$, and $k_{i-1,j}$ for $i,j\geq 3$.

\section{Throw-away eigenvectors}
\label{se:throwaway_eigenvectors}

We have found that the speed of convergence can be improved
significantly by performing the preconditioned inverse iteration
not only for the $p$ vectors we are actually interested in,
but for a few more ``throw-away eigenvectors'' that only serve
to speed up the rate of convergence.

To motivate this approach, we consider the basic block inverse
iteration
\begin{align*}
  E^{(m+1)} &= A^{-1} E^{(m)} &
  &\text{ for all } m\in\bbbn_0,
\end{align*}
where $E^{(m)}\in\bbbc^{n\times p}$ and $A\in\bbbc^{n\times n}$
is a self-adjoint positive definite matrix.
Since $A$ is self-adjoint, we can find a unitary matrix
$Q\in\bbbc^{n\times n}$ and eigenvalues $0<\lambda_1\leq\lambda_2
\leq\ldots\leq\lambda_n$ with
\begin{equation*}
  Q^* A Q = D = \begin{pmatrix}
    \lambda_1 & & \\
    & \ddots & \\
    & & \lambda_n
  \end{pmatrix}.
\end{equation*}
For the transformed iterates
\begin{align*}
  \widehat{E}^{(m)} &:= Q^* E^{(m)} &
  &\text{ for all } m\in\bbbn_0
\end{align*}
we have
\begin{align*}
  \widehat{E}^{(m+1)} &= D^{-1} \widehat{E}^{(m)} &
  &\text{ for all } m\in\bbbn_0.
\end{align*}
We define
\begin{align*}
  \begin{pmatrix}
    F\\ R
  \end{pmatrix} &= \widehat{E}^{(0)}, &
  F &\in\bbbc^{p\times p},\ R\in\bbbc^{(n-p)\times p},\\
  \begin{pmatrix}
    D_p & \\
    & D_\perp
  \end{pmatrix} &= D, &
  D_p &\in\bbbc^{p\times p},\ D_\perp\in\bbbc^{(n-p)\times(n-p)}
\end{align*}
and observe
\begin{align*}
  \widehat{E}^{(m)}
  &= D^{-m} \widehat{E}^{(0)}
   = \begin{pmatrix}
       D_p^{-m} F\\
       D_\perp^{-m} R
     \end{pmatrix} &
  &\text{ for all } m\in\bbbn_0.
\end{align*}
Since we want to approximate a basis for the invariant subspace
spanned by the first $p$ eigenvectors, we have to assume that $F$
has full rank, i.e., that it has to be invertible.
This assumption leads to
\begin{align*}
  \widehat{E}^{(m)} F^{-1}
  &= \begin{pmatrix}
       D_p^{-m}\\
       D_\perp^{-m} R F^{-1}
     \end{pmatrix} &
  &\text{ for all } m\in\bbbn_0.
\end{align*}
Let $i\in[1:p]$ and denote the $i$-th columns of these matrices
by $e_i^{(m)}$ and the $i$-th canonical unit vector by $\delta_i$.
Due to $\|D_\perp^{-m}\| \leq |\lambda_{p+1}|^{-m}$, our equation implies
\begin{align*}
  \tan\angle(e_i^{(m)}, \delta_i)
  &\leq C \left(\frac{|\lambda_i|}{|\lambda_{p+1}|}\right)^m &
  &\text{ for all } m\in\bbbn_0,
\end{align*}
i.e., the $i$-th column of $\widehat{E}^{(m)}$ converges to
the $i$-th eigenvector of $D$ at a rate of $|\lambda_i|/|\lambda_{p+1}|$,
and therefore the $i$-th column of $E^{(m)}$ converges to
the $i$-th eigenvector of $A$ at the same rate.
This means that we can expect the convergence rate to improve
if we increase $p$.

%
%
\begin{figure}
\begin{center}
  \includegraphics[width=0.8\textwidth]{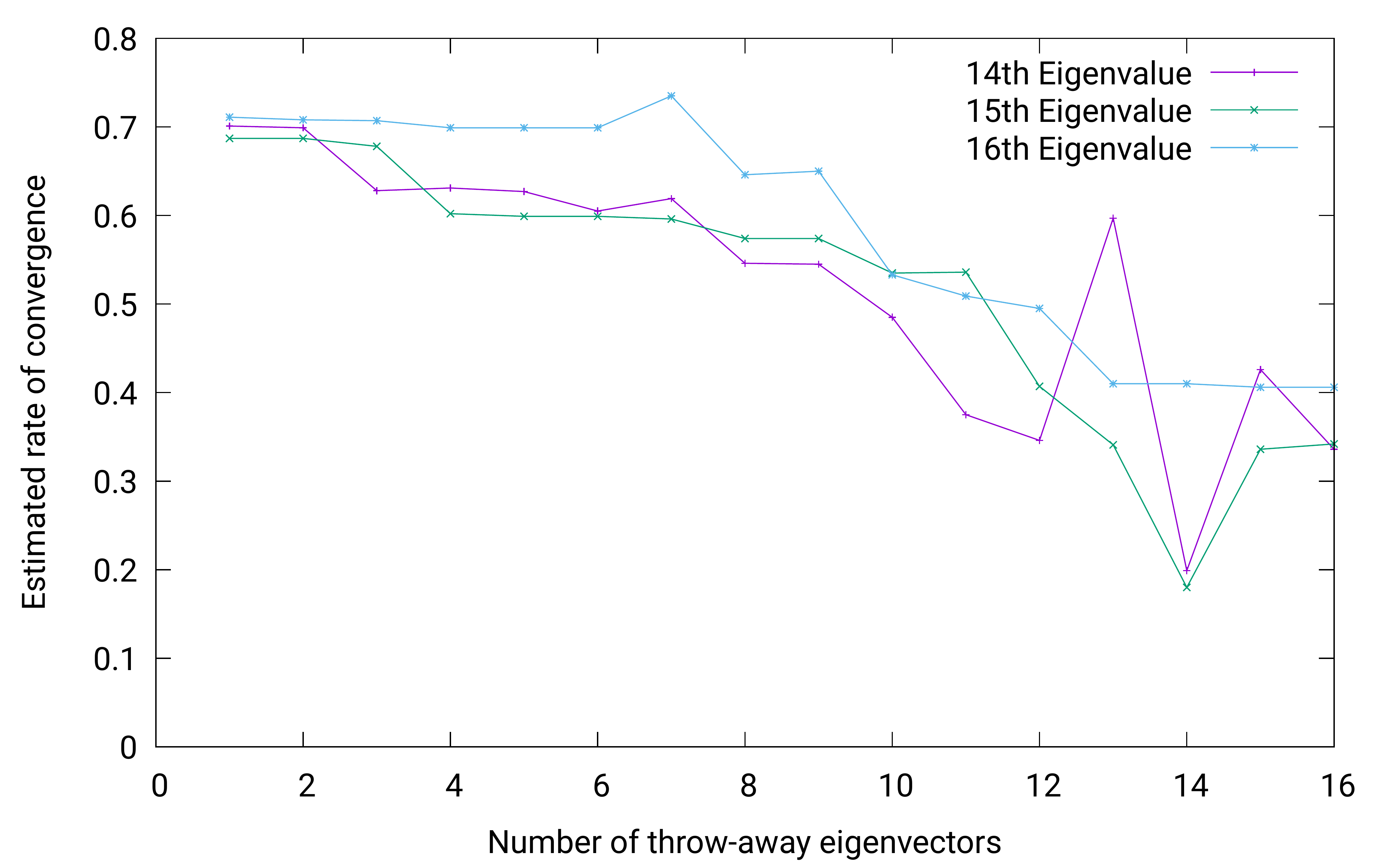}
\end{center}
\caption{Experimentally-observed rates of convergence for the
  $14$th, $15$th, and $16$th non-zero eigenvalues depending on
  the number of throw-away eigenvectors}
\label{fi:throwaway}
\end{figure}

In our implementation, we use the \emph{preconditioned} block
inverse iteration and orthonormalize the iterates after every
step.
\Cref{fi:throwaway} shows the results of an experiment with
the permittivity
\begin{align*}
  \epsilon(x) &= \begin{cases}
    100 &\text{ if } \|x-(1/2,1/2)\|\leq 1/3,\\
    1 &\text{ otherwise}
  \end{cases} &
  &\text{ for all } x\in[0,1]\times[0,1],
\end{align*}
where we aim to compute the first $16$ non-zero eigenvalues and
add between $1$ and $16$ further ``throw-away'' eigenvectors to
speed up convergence.
We can see that the experimentally observed rate of convergence
indeed is improved by adding more eigenvectors.

Of course, computing more eigenvectors increases the computational
work (for $p\ll n$, we expect to need $\mathcal{O}(n p^2)$ operations),
but our experiments indicate that the impact is more than compensated by the
decrease in the number of required iteration steps if we base our
stopping criterion only on the convergence of the relevant eigenvectors.

\section{Numerical experiments}
\label{se:numerical_experiments}

Since we can expect the eigenvalues to depend smoothly on
the Bloch parameter, we can replace the entire Bloch parameter
set $[-\pi/a,\pi/a]\times[-\pi/b,\pi/b]$ by a sufficiently fine equidistant
grid.
For our experiment, we choose $a=b=1$ and a grid with $30\times 30$
points

Maxwell's equation is discretized on a coarse grid on the domain
$[0,1]\times[0,1]$ with $16\times 16$ square elements on the coarsest
mesh and $1024\times 1024$ square elements on the finest.
On the finest mesh, we therefore have $2\,097\,152$ N\'ed\'elec
basis functions.

We choose the piecewise constant permittivity function
\begin{align*}
  \epsilon(x) &= \begin{cases}
    11.56 &\text{ if } \|x-(0.5,0.5)\|\leq 0.18,\\
    1 &\text{ otherwise}
  \end{cases} &
  &\text{ for all } x\in[0,1]\times[0,1]
\end{align*}
introduced by \cite{MIKI03}.

We compute the first $16$ eigenvalues using the block preconditioned
inverse iteration.
We stop the iteration as soon as the defects
$\|A e^{(m)} - M e^{(m)} \lambda_m\|_2$ of all eigenvector approximations
drops below $10^{-2}$.
Considering the scaling behaviour of the spectral norm as the grid
is refined, this accuracy has been sufficient in our experiments.

In order to improve the rate of convergence, we compute $8$ additional
``throw-away'' eigenvector approximations, but they serve only to speed up
convergence and to allow the algorithm to choose the approximations of the
first $16$ eigenvectors from a $24$-dimensional space, they are not considered
for the stopping criterion.
The first and second eigenvalues depending on the Bloch parameter $k$
are displayed in \cref{fi:bloch_eigenvalues1}, the
third and fourth in \cref{fi:bloch_eigenvalues2}.
We observe that the eigenvalues change smoothly with the Bloch
parameter.
For all eigenvectors, the stopping criterion was reached, so the
computed vectors are good approximations of the exact eigenvectors.

%
%
\begin{figure}
\begin{center}
  \includegraphics[width=0.8\textwidth]{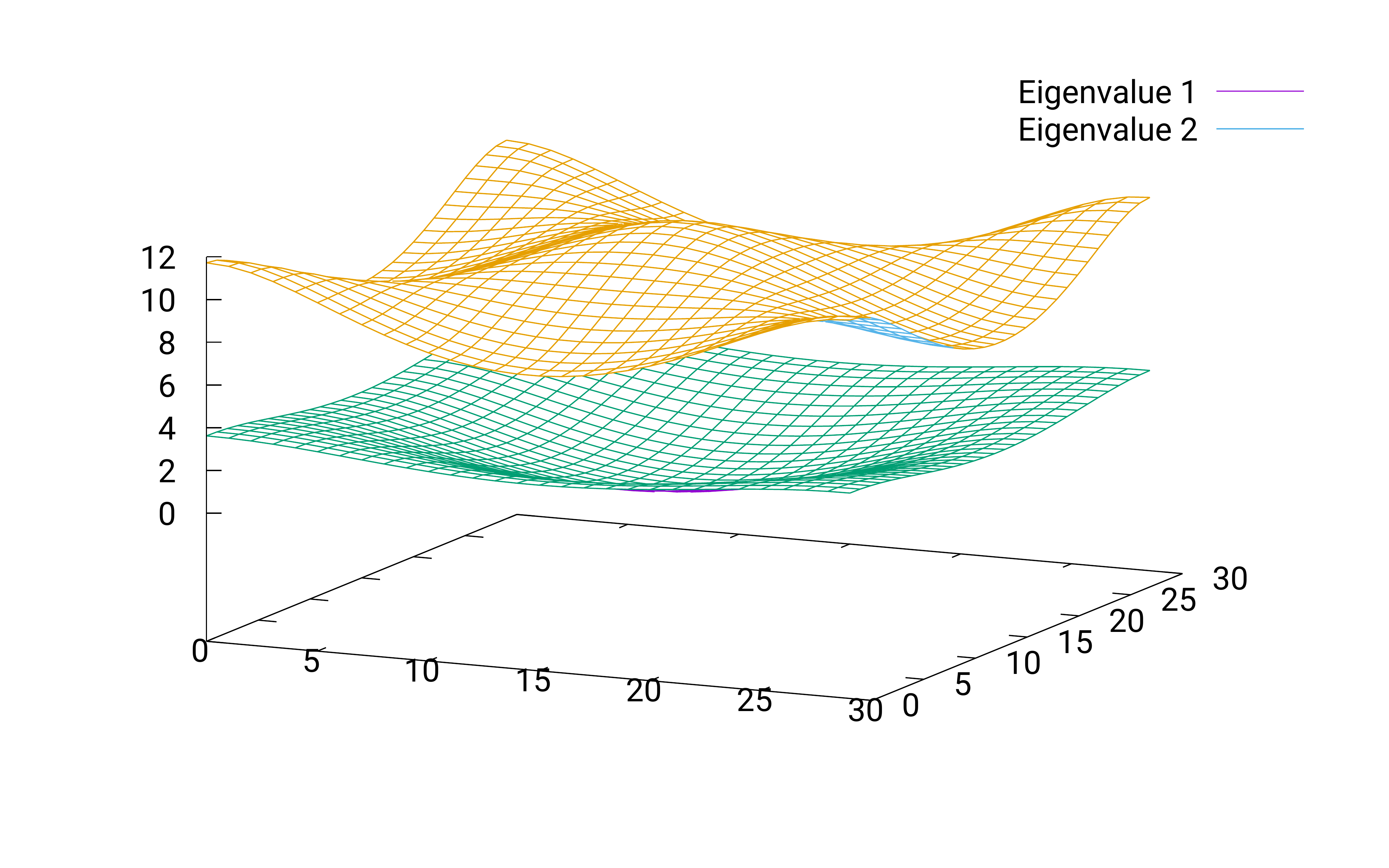}
\end{center}
\caption{First and second eigenvalues depending on the Bloch parameter}
\label{fi:bloch_eigenvalues1}
\end{figure}

%
%
\begin{figure}
\begin{center}
  \includegraphics[width=0.8\textwidth]{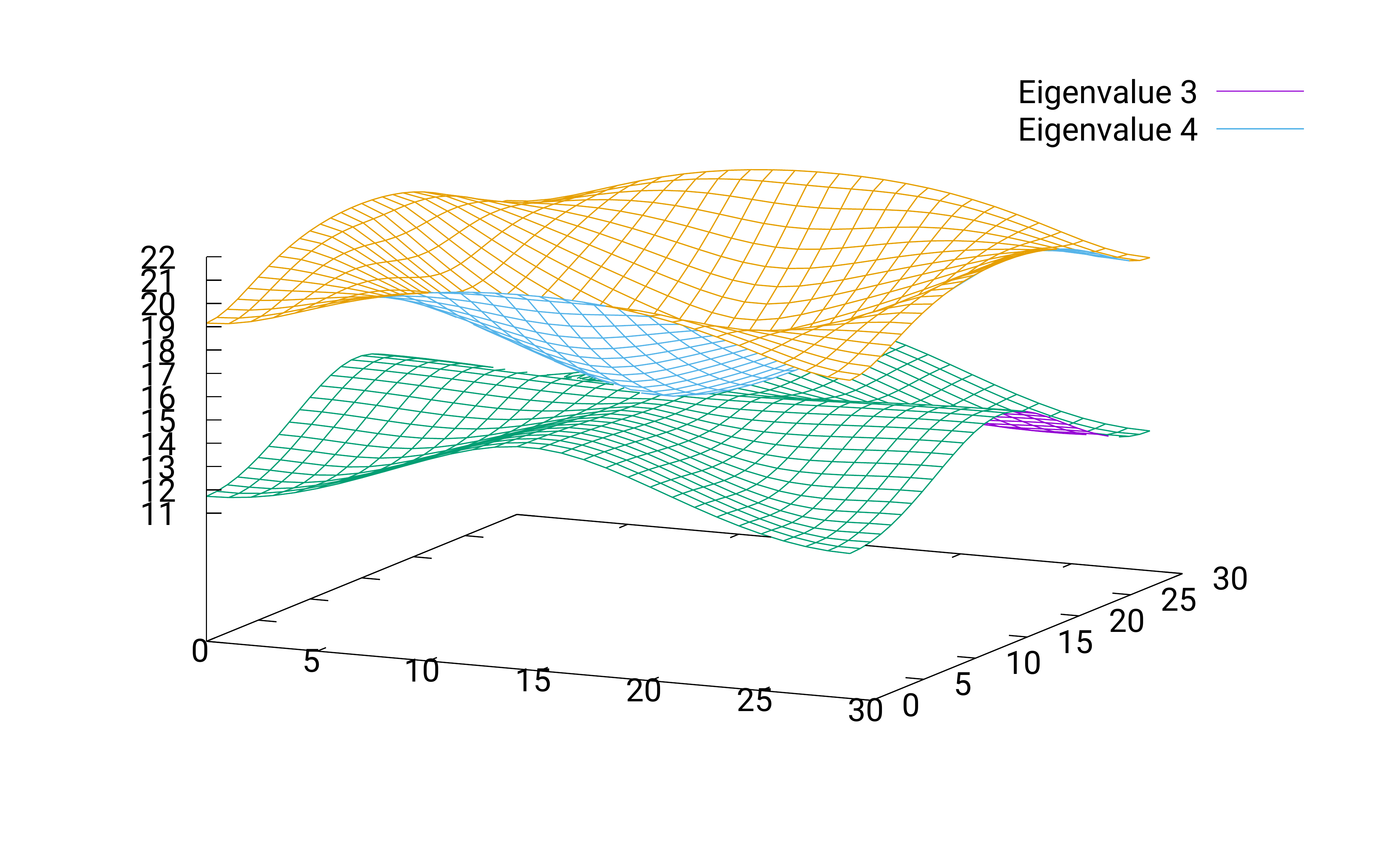}
\end{center}
\caption{Third and fourth eigenvalues depending on the Bloch parameter}
\label{fi:bloch_eigenvalues2}
\end{figure}

Of course, we are interested in the numerical performance of our
method.
\Cref{fi:bloch_iterations} shows the number of iterations required
for the different Bloch values.
We can see that our extrapolation method works very well:
extrapolating between adjacent Bloch values to obtain an initial
guess for the preconditioned inverse iteration reduces the number
of required iteration steps to less than $4$ in most of the
cases.
Only close to the special case $k=0$ (in the four corners of the
diagram), the algorithm requires a significantly increased number
of steps.

%
%
\begin{figure}
\begin{center}
  \includegraphics[width=0.8\textwidth]{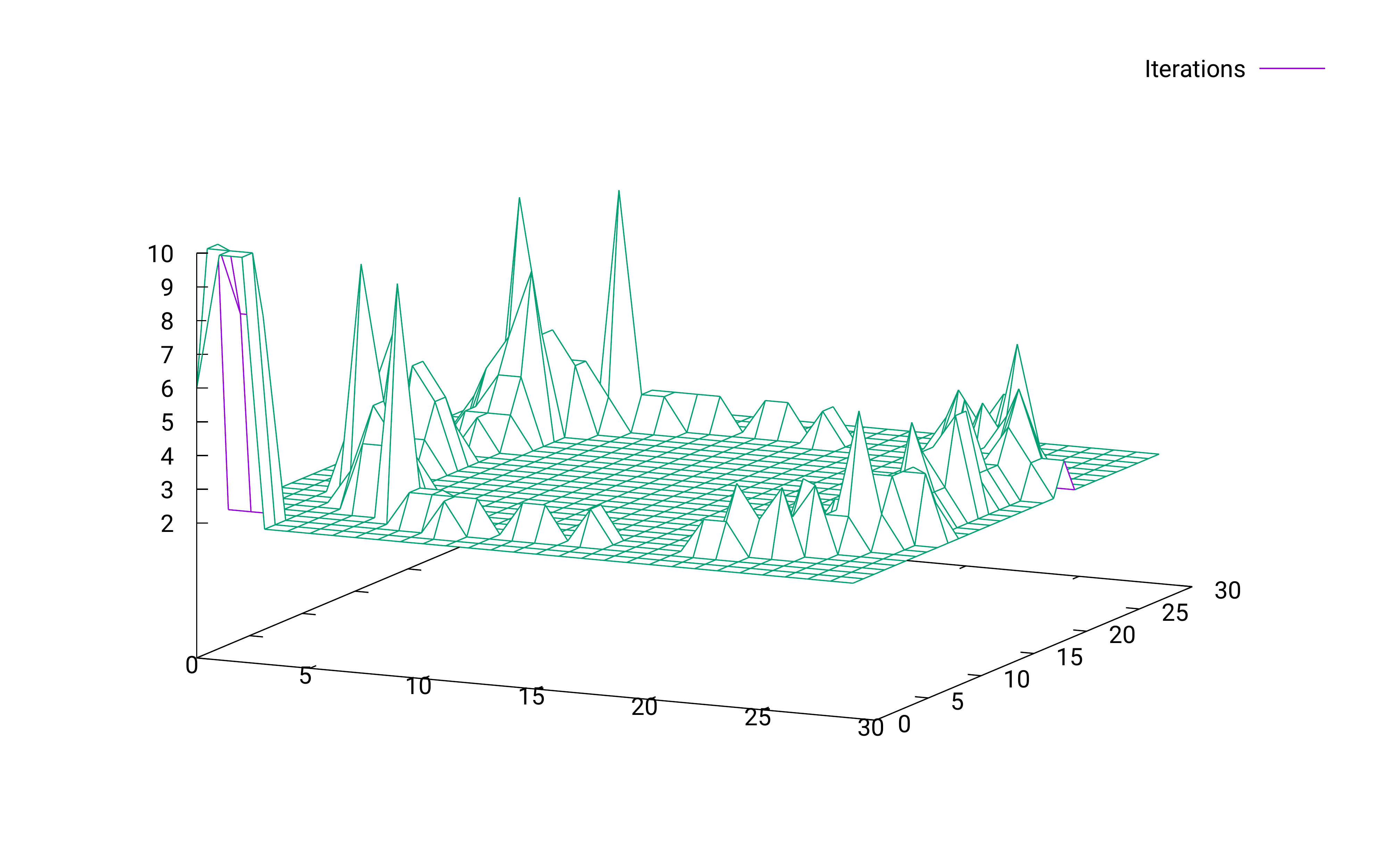}
\end{center}

  \caption{Iterations required to obtain a residual norm
    below $10^{-2}$}
  \label{fi:bloch_iterations}
\end{figure}


\section*{Acknowledgments}
We would like to acknowledge the support of the Kiel Nano and
Interface Science (KiNSIS) initiative.

\bibliographystyle{plain}
\bibliography{scicomp}

\end{document}